\documentclass[11pt]{article}

\usepackage{amssymb}
\usepackage{amsmath}
\usepackage{amsthm}
\usepackage{mathrsfs}
\usepackage{graphicx}
\usepackage[utf8]{inputenc}
\usepackage{multirow}
\usepackage{comment}
\usepackage{tikz}
\usetikzlibrary{tikzmark}

\newtheorem{cor}{Corollary}
\newtheorem{rem}{Remark}
\newtheorem{ex}{Example}

\def\mod{\!\pmod}

\title{Improved Bounds on  Diffsequences with Gaps in Powers of 2}
\author{Kanav Talwar\thanks{
                  Delhi Public School Faridabad, Faridabad, HR 121002, India}
        \and
        Utkarsh Gupta\thanks{
                      Department of Electrical and Computer Engineering, Northeastern University, Boston, MA 02115, USA
                  }
        }
\usepackage{amssymb}
\usepackage{amsmath}
\usepackage{amsthm}
\usepackage{mathrsfs}
\usepackage{graphicx}
\usepackage[utf8]{inputenc}
\usepackage{multirow}
\usepackage{color}
\newtheorem{theorem}{Theorem}
\newtheorem{lemma}{Lemma}
\newtheorem{proposition}{Proposition}
\newtheorem{definition}{Definition}

\begin{document}
\newpage
\maketitle
\begin{abstract}
    Let $D$ be a set of positive integers. A $D$-diffsequence of length $k$ is a sequence of positive integers $a_1 <  \cdots < a_k$ such that $a_{i+1}-a_i\in D$ for $i=1,\ldots,k-1$. For $D=\{2^i\mid i\in \mathbb{Z}_{\ge 0}\}$, it is known that there exists a minimum integer $n$, denoted by $\Delta(D,k)$, such that every $2$-coloring of $\{1,\ldots n \}$ admits a monochromatic $D$-diffsequence of length $k$. In this work, we prove a new lower bound for $\Delta(D,k)$ to $\Delta(D,k)\ge \left(\sqrt{\frac{8k-5}{12}}-\frac12\right)2^{\left(\sqrt{\frac{8k-5}{3}}-3\right)}$, asymptotically improving the exponential constant in the bound proved by Clifton.
\end{abstract}



\section{Introduction}
An $r$-coloring of a set $A$ is a function $\chi : A \to \{0, \ldots, r-1\}$, where each element of $A$ is assigned one of $r$ distinct colors. We say a subset $A' \subseteq A$ is monochromatic under $\chi$ if $\chi$ is constant over $A'$. In 1927, Van der Waerden~\cite{Waerdan} showed that for any $r$-coloring of the positive integers, at least one color is such that there are monochromatic arithmetic progressions of arbitrary lengths. Equivalently, for any positive integer $k$, there exists a minimum integer $n$, such that for any $r$-coloring of $\{1,\ldots, n \}$, there must be a monochromatic arithmetic progression of length $k$. Motivated by generalizing this property to other sequences, Landman and Robertson~\cite{LR} considered the following question: Does an $r$-coloring of positive integers produce a monochromatic sequence of arbitrary length, whose gaps lie in some fixed set $D$? In particular, they introduced the notion of a $D$-diffsequence.
\begin{definition}
    Given a set $D$ of positive integers, a $D$-\textit{diffsequence} of length $k$ is a sequence of positive integers $a_1 < \dots < a_k$ such that $a_{i+1}-a_i\in D$ for $i \in \{1,\dots,k-1\}$.
\end{definition}
\begin{definition}
    A set $D$ is called $r$-accessible if every $r$-coloring of the positive integers contains arbitrarily long monochromatic $D$-diffsequences.
\end{definition}
Several works have sought to characterize the accessibility of canonical integer sets $D$ (see \cite{book}), such as the set of Fibonacci numbers \cite{LR, Wesley}, primes \cite{LR}, and fixed translates of primes \cite{LV10}, and powers of integers. Consider the set of powers of some integer $t \in \mathbb{Z}_{> 0}$, i.e. $D = \{ t^{i}   \mid i \in \mathbb{Z}_{\ge 0}\}$. For $t\ge 3$, there is a periodic coloring modulo $t-1$ that avoids arbitrarily long $D$-diffsequences. It is known that $D$ is only $2$-accessible when $t = 2$ \cite{book}. Analogous to Van der Waerden's inquiry, the problem of $r$-accessibility of a set $D$ can be rephrased to finding the minimum integer $n$, denoted by $\Delta(D,k;r)$, such that every $r$-coloring of $\{1, \ldots, n\}$ admits a monochromatic $D$-diffsequence of length $k$. 

\begin{definition}
    Let $\Delta(D,k;r)$ be the smallest positive integer $n$, such that for any $r$-coloring of $\{1, \ldots, n \}$, we can find a $k$ length monochromatic $D$-diffsequence. Let $\Delta(D,k) = \Delta(D,k;2)$.
\end{definition}

For the set $D$ of powers of $2$, i.e., $D = \{ 2^i \mid i \in \mathbb{Z}_{\ge 0}\}$, Landman and Robertson \cite{LR} showed that $D$ is $2$-accessible, with $\Delta(D,k) \le 2^k - 1$. Chokshi et. al. \cite{CCLS18} conjectured that $\Delta(D,k)$ can be lower bounded exponentially. Clifton \cite{Clifton} established a lower bound which is asymptotically $\mathcal{O}(2^{\sqrt{2k}})$. In particular,
\begin{equation}
    \Delta(D, k) \geq 2^{\sqrt{2k}} \left( \frac{(\sqrt{2} - 1)k}{8} - \frac{\sqrt{k}}{8} \right) + \frac{\sqrt{k}}{2}.
\end{equation}
In this work, we establish a new lower bound for $\Delta(D,k)$, asymptotically improving the exponential constant in Clifton's bound from $\mathcal{O}(2^{\sqrt{2k}})$ to $\mathcal{O}(2^{\sqrt{8k/3}})$. Our main result is stated as follows.  
\begin{theorem}\label{thm: MainThm}
    For $D=\{2^i\mid i\in\mathbb{Z}_{\ge0}\}$ we have
    \begin{equation}\label{Theone}
        \Delta(D,k)\ge \left(\sqrt{\frac{8k-5}{12}}-\frac12\right)2^{\left(\sqrt{\frac{8k-5}{3}}-3\right)}.
    \end{equation}
\end{theorem}

\section{Construction}\label{sec: construction}
We will begin by formally constructing the coloring function $\kappa$ used to show the bound in Theorem~\ref{thm: MainThm}. We define a mapping from some coloring of $l$ integers to a coloring of $4l$ integers, which we call an \textit{expansion}. The process of \textit{expansion} is then recursively used to construct the required coloring $\kappa$. To prove the lower bound in Theorem~\ref{thm: MainThm}, we explicitly define a $2$-coloring scheme of $\{1,\ldots, l \cdot  4^{l-1}\}$ which avoids any monochromatic diffsequences of length $k = (3l^2-3l+2)/2$.  For a coloring $\chi \colon\{1,\ldots, l\} \to \{0,1\}$, define its \textit{expansion} $\chi^{(1)}\colon \{1,\ldots, 4l\} \to \{0,1\}$ by setting, for $i \in \{1, \ldots,  l\}$ and $j \in [0,3]$, 
\begin{equation}
\chi^{(1)}(4i - j) =
\begin{cases}
\chi(i), & \text{if } j \in \{2,3\}, \\
1 - \chi(i), & \text{if } j \in \{0,1\}.
\end{cases}
\end{equation}
The coloring scheme is constructed recursively using the process of expansion. Repeating this procedure $r$ times, we eventually obtain a coloring of the first $l \cdot 2^{2r}$ positive integers, which we call $\chi^{(r)}$. To show the lower bound in Theorem~\ref{thm: MainThm}, begin with the trivial coloring: $\chi_{\mathrm{alt}} \colon \{1,\ldots, l\} \to \{0,1\}$ such that $\chi_{\mathrm{alt}}(i) = i \mod{2}$. Then, we will show that the coloring $\chi_{\mathrm{alt}}^{(l-1)} \colon [ l \, \cdot \, 4^{l-1}] \to \{0,1\}$ avoids monochromatic diffsequences of length $k$, i.e. $\kappa \equiv  \chi_{\mathrm{alt}}^{(l-1)}$.

\section{Preliminaries}\label{prev}
In this section, we will rephrase $2$-coloring functions, and the process of expansion to their more intuitive analogue, binary strings. For positive integers \( N \), by \(  [N] \) we mean the set \( \{1, 2, 3, \ldots, N\} \). By a 2-coloring of \([N] \) we mean a mapping \( \chi :[N] \to \{0, 1\} \). For the rest of this work, we will use the notation $D$ to denote the set of all non-negative integer powers of two, i.e., $D=\{2^i\mid i\in\mathbb{Z}_{\ge0}\}$.
Note that any $2$-coloring of $[n]$ can be represented as a binary string $\mathbf{s} = s_1 \ldots s_n$, where $s_i = \chi(i)$. For a string $\mathbf{s}$ corresponding to a given $2$-coloring of $[n]$, a monochromatic diffsequence $\mathcal{A} = \{a_i \}_{i=1}^{k}$ is such that $a_{i+1}-a_i \in D$ and $s_{a_{i+1}} = s_{a_i}$, i.e., $\chi(a_{i+1}) = \chi(a_i)$, for $i \in [k-1]$.
\begin{definition}\label{expand}
    Let $\mathbf{s}$ be a binary string of length $l$. The expansion of $\mathbf{s}$ is defined as a binary string $\mathbf{s}^{(1)}$ of length $4l$, obtained by replacing each $0$ in $\mathbf{s}$ by $0011$ and each $1$ in $\mathbf{s}$ by $1100$. The string $\mathbf{s}^{(r)}$ is recursively defined as the expansion of $\mathbf{s}^{(r-1)}$, where $r \ge 1$ and $\mathbf{s}^{(0)} = \mathbf{s}$.
\end{definition}


\begin{ex}\label{bigex}
    Consider the string $\mathbf{t}=1011$, its expansion is $\mathbf{t}^{(1)} = \overbrace{1100}_{\phantom{.}}\,\overbrace{0011}_{\phantom{.}}\,\overbrace{1100}_{\phantom{.}}\,\overbrace{1100}_{\phantom{.}}$. 
\end{ex}
In Lemma~\ref{MainLemma}, we will bound the length of certain diffsequences of an expanded string. Therefore, we wish to understand how diffsequences behave under expansion. Let $\mathbf{s}^{(1)}$ be the expansion of some string $\mathbf{s}$. The process of expansion is such that the $i^{th}$ bit of $\mathbf{s}^{(1)}$ is determined uniquely by the $\lceil i/4 \rceil^{th}$ bit of the string $\mathbf{s}$. This lets us determine the \textit{block}, either $0011$ or $1100$, to which the $i^{th}$ bit of $\mathbf{s}^{(1)}$ belongs. For convenience, we call this mapping the position of $i$.
\begin{definition}\label{pos}
    Let the position of $i \in \mathbb{Z}_{+}$ be defined as the function $\text{pos}(i)=\left\lceil i/4 \right\rceil$.
\end{definition}


The choice of $0011$ and $1100$ as the types of blocks used in expansion is motivated by the following two reasons: a) The length of blocks being $4$ implies that for any two bit indices $i$ and $i'$ in the expanded string $\mathbf{s}^{(1)}$ that differ by a power of $2$, i.e. $i'-i \in D$, their corresponding positions in the string $\mathbf{s}$ are such that $\mathrm{pos}(i') - \mathrm{pos}(i) \in D \, \cup \, \{0\}$; and b) For a monochromatic diffsequence $a_1< \ldots < a_k$ in the expanded string which is such that some of its consecutive elements lie in different blocks i.e. $s_{pos(a_{j+1})} \neq s_{pos(a_{j})}$, for some $j \in [k-1]$, we must have that $a_{j+1} - a_{j} \le 2$, i.e. they must lie in consecutive blocks. This is enabled by the fact that the blocks $0011$ and $1100$ are conjugates of each other. This will let us reduce monochromatic diffsequences in the expanded string $\mathbf{s}^{(1)}$ to two contiguous diffsequences of opposite colors in $\mathbf{s}$, partitioned by a $\textit{hop}$. A hop is defined to be an index in a diffsequence, at which the assigned color of the diffsequence changes.

\begin{definition}\label{def: hops}
Let $\mathbf{s} \in \{ 0,1 \}^n$ and $\mathcal{A}=\{a_i\}_{i=1}^k$ a $D$-diffsequence. A $\textit{hop}$ is defined as an index $j \in [k-1]$ such that $a_{j+1}-a_j=1$ and $s_{a_{j+1}}\neq s_{a_{j}}$. Define $H_{\mathcal{A}}$ to be the set of all hops in $\mathcal{A}$.
\end{definition}

\begin{definition}\label{def: Psi_psi}
For a string $\mathbf{s} \in \{ 0,1 \}^n$, define $\Psi_{\mathbf{s}}(h)$ to be the set of all diffsequences with at most $h$ hops. Formally, $ \Psi_{\mathbf{s}}(h): =\{\mathcal{A}=\{a_i\}_{i=1}^k 
    :\mathcal{A} \text{ is a } D\text{-diffsequence with } |H_{\mathcal{A}}|\le h\}$. Let $\psi_\mathbf{s}(h)$ denote the maximum length of any diffsequence in $\Psi_{\mathbf{s}}(h)$, and $\psi_\mathbf{s}(h) = 0$ if $\Psi_{\mathbf{s}}(h) = \emptyset$.
\end{definition}

\begin{rem}
    $\Psi_{\mathbf{s}}(0)$ is the set of all monochromatic diffsequences of $\mathbf{s}$. Therefore, if the length of $\mathbf{s}$ is greater than $\Delta(D,k)$, the maximum length of elements of $\Psi_{\mathbf{s}}(0)$ is at least $k$.
\end{rem}

\begin{ex}\label{hop example}
Consider the string $\mathbf{s}=1100001100111100$ of length $16$. The sequence of integers $\mathcal{A} = \{1,2,3,5,9\}$ is an example of a $D$-diffsequence. The corresponding subsequence of $\mathbf{s}$ is $s_1s_2s_3s_5s_9 = 11000$, and $H_{\mathcal{A}} = \{2\}$. Since this is the only hop in the diffsequence, $\mathcal{A} \in \Psi_{\mathbf{s}}(1)$.   
\end{ex}

The following proposition and its corollaries give us a recipe to reduce diffsequences with $h$ hops in some expanded string $\mathbf{s}^{(1)}$, to diffsequences with $h+1$ hops in $\mathbf{s}$. Proposition~\ref{prop: reduce_diffseq} shows that in the expanded string $\mathbf{s}^{(1)}$, if two positions $\{x,y\}$ of some diffsequence have the same color $c \in \{0, 1 \}$, and their corresponding positions $\{\text{pos}(x), \text{pos}(y) \}$ in $\mathbf{s}$ have different colors, then $\text{pos}(y)=\text{pos}(x)+1$ i.e. $x$ and $y$ lie in blocks next to each other. Moreover, we must have $s_{\text{pos}(x)}=1-c$, and $ s_{\text{pos}(y)}=c$. 

\begin{proposition}\label{prop: reduce_diffseq}
    For a binary string $\mathbf{s}$,
    let $x,y$ be positive integers such that $y-x\in D$ and $s^{(1)}_x=s^{(1)}_y$. If $s_{\text{pos}(x)}\neq s_{\text{pos}(y)}$, we must have $s_{\text{pos}(x)}=1-s^{(1)}_x$, $ s_{\text{pos}(y)}=s^{(1)}_y$, and $\text{pos}(y)=\text{pos}(x)+1$.
\end{proposition}

\begin{proof}
The condition $s_{\text{pos}(x)}\neq s_{\text{pos}(y)}$ means that $s_x^{(1)}$ and $s_y^{(1)}$ belong to different types of blocks, $0011$ and $1100$. Without loss of generality, assume that $s_x^{(1)} = s_y^{(1)}$ = 1. The proof is divided into two cases according to the value of $y-x$. (a) For $y-x \in \{1,2 \}$, we have $\text{pos}(y)-\text{pos}(x)\in \{0, 1\}$ which means that $s_x^{(1)}$ and $s_y^{(1)}$ belong either to the same block or consecutive blocks. However, since $s_x^{(1)}$ and $s_y^{(1)}$ must belong to different types of blocks, these must be consecutive.
Therefore $\text{pos}(y)-\text{pos}(x)=1$. The condition $s_x^{(1)} = s_y^{(1)} = 1$ is only possible when 
$s_x^{(1)} \in \{0011\}$, and $s_y^{(1)} \in \{1100\}$, which implies $s_{\text{pos}(x)}=0$ and $s_{\text{pos}(y)}=1$. (b) For $y-x=2^m$, $x \equiv y \mod{4}$ and therefore, $s_x^{(1)}$ and $s_y^{(1)}$ occupy the same position within their respective blocks. Since the two types of blocks, $0011$ and $1100$ are conjugates of each other, $s^{(1)}_{x} = s^{(1)}_{y}$ is only possible when the blocks are identical. This would mean that $s_{\text{pos}(x)}$ and $s_{\text{pos}(y)}$ are identical, which is a contradiction.
\end{proof}

 


Let $x_1<\dots<x_k$ be a monochromatic diffsequence of color $c \in \{0,1\}$ in the string $\mathbf{s}^{(1)}$, which is the expansion of some string $\mathbf{s}$. As noted earlier, $y-x \in D$, implies $\text{pos}(y)-\text{pos}(x) \in D \ \cup \{0\}$; the sequence characterized by the set $\{\text{pos}(x_i)\}_{i=1}^{k}$ is a $D$-diffsquence in $\mathbf{s}$. A consequence of Proposition~\ref{prop: reduce_diffseq} is that the substring $s_{\text{pos}(x_i)}$ can be split into two contiguous monochromatic substrings, the first with the color $1-c$ and the other with $c$. This result is proved formally in Corollary~\ref{A B split cor}. Therefore, a monochromatic diffsequence in $\mathbf{s}^{(1)}$ can be mapped to a diffsequence with one switch in color in $\mathbf{s}$, i.e., with one hop (Definition~\ref{def: hops}). This is a mapping between elements of $\Psi_{\mathbf{s^{(1)}}}(0)$ to elements of $\Psi_{\mathbf{s}}(1)$. In Corollary~\ref{pos bound cor}, we will establish that the set $\{\text{pos}(x_i)\}_{i=1}^{k}$ has at least $k-1$ distinct elements under certain conditions. These two corollaries together help in bounding $\psi_{\mathbf{s^{(1)}}}(0)$ in terms of $\psi_{\mathbf{s}}(1)$ (Definition~\ref{def: Psi_psi}) as $\psi_{\mathbf{s^{(1)}}}(0) \le \psi_{\mathbf{s}}(1) +2$.

\begin{cor}\label{A B split cor}
    Let $\mathbf{s}^{(1)}$ be an expansion of a binary string $\mathbf{s}$ and $x_1<\dots<x_k$ be a monochromatic $D$-diffsequence in $\mathbf{s}^{(1)}$ with color $c \in \{0,1\}$. Then, there exists a non negative integer $m\le k$ such that $s_{\text{pos}(x_j)}=1-c$ if and only if $j\le m$.
\end{cor}
\begin{proof}
    From Proposition~\ref{prop: reduce_diffseq} we have the following possibilities for each $i\le k-1$: $s_{\text{pos}(x_i)}=s_{\text{pos}(x_{i+1})}$ or $s_{\text{pos}(x_i)}=1-c$ and $s_{\text{pos}(x_{i+1})}=c$. Now, if there exists an index $j$ for which $s_{\text{pos}(x_j)}=1-c$ and $s_{\text{pos}(x_{j+1})}=c$ (if not, then we are already done), then pick the smallest such $j$. Thus, $s_{\text{pos}(x_i)}=1-c$ for all $i\le j$. We must have $s_{\text{pos}(x_i)}=c$ for all $i>j$, as otherwise there would exist an index $j^{\prime}$ such that $s_{\text{pos}(x_{j^{\prime}})}=c$ and $s_{\text{pos}(x_{j^{\prime}+1})}=1-c$, which is not possible. Hence, proving our claim.
\end{proof}

\begin{cor}\label{pos bound cor}
   Let $\mathbf{s}^{(1)}$ be an expansion of a binary string $\mathbf{s}$ and $x_1<\dots<x_k$ be a monochromatic $D$-diffsequence in $\mathbf{s}^{(1)}$ with color $c\in\{0,1\}$. Suppose that there exists $d\in\{0,1\}$ such that $s_{\text{pos}(x_i)}=d$ for all $1\le i\le k$. Then the set $\{\text{pos}(x_i)\}^k_{i=1}$ has at least $k-1$ distinct elements.
\end{cor}
\begin{proof}
Without loss of generality, assume $c=1$. The proofs for $d=0$ and $d=1$ are analogous; therefore, we will further assume $ d=0$. Since $\{\text{pos}(x_i)\}^k_{i=1}$ is a non-decreasing sequence, to prove it has at least $k-1$ distinct elements, it suffices to show that $\text{pos}(x_j)=\text{pos}(x_{j+1})$ is true for at most one index. If it exists, let $j$ be such an index. For each $1\le i\le k$, the conditions $s^{(1)}_{x_i}=1$ and $s_{\text{pos}(x_i)}=0$ imply that $s^{(1)}_{x_i}$ corresponds to the third or the fourth position in the block $0011$. Furthermore, if $s^{(1)}_{x_i}$ corresponds to the fourth position in the block $0011$, then $s^{(1)}_{x_{i+1}}$ must also correspond to the fourth position. This is because if it corresponds to the third position, then the difference between $x_{i+1}$ and $x_i$ would be odd, implying it to be $1$, but this is clearly not possible. For an index $j$ to satisfy $\text{pos}(x_j)=\text{pos}(x_{j+1})$, we must have $s^{(1)}_{x_j}$ and $s^{(1)}_{x_{j+1}}$ correspond to the third and fourth position in the block $0011$ respectively. This means that all indices at least $j+1$ must correspond to the fourth position in the block, thus enforcing the impossibility of another index $i$ satisfying $\text{pos}(x_i)=\text{pos}(x_{i+1})$.
\end{proof}


Corollaries \ref{A B split cor} and \ref{pos bound cor} show that the process of expansion naturally induces a mapping from monochromatic diffsequences in an expanded string and diffsequences with one hop in the original string. Motivated by this approach, in Lemma~\ref{MainLemma}, we inductively extend this reasoning to obtain a mapping from elements in $\Psi_{\mathbf{s}^{(1)}}(h)$ to $\Psi_{\mathbf{s}}(h+1)$ for some string $\mathbf{s}$. This will help us obtain a bound on the maximum length of a monochromatic diffsequence in our explicit construction.


\section{Main result}\label{MainResult}
To prove the main result of this work, Theorem~\ref{thm: MainThm}, we explicitly need to show that the construction in Section~\ref{sec: construction} does not admit any monochromatic $D$-diffsequences of length $k$. For any string $\mathbf{s}$ and its expansion $\mathbf{s}^{(1)}$, Lemma~\ref{MainLemma} gives a bound on the difference between the maximum length of a diffsequence with at most $h+1$ hops in $\mathbf{s}$, $\psi_{\mathbf{s}}(h+1)$, and the
maximum length of a diffsequence with $h$ hops in the expanded string $\mathbf{s}^{(1)}$, $\psi_{\mathbf{s}^{(1)}}(h)$ (Definition~\ref{def: Psi_psi}). Formally, for any binary string $\mathbf{s}$, Lemma~\ref{MainLemma} shows that $\psi_{\mathbf{s}^{(1)}}(h) \le \psi_{\mathbf{s}}(h+1) + 3h + 2$. Recall that in our explicit construction, we begin with the initial coloring $\chi_{\text{alt}} \colon [l] \to \{0,1\}$ such that $\chi_\text{alt}(i) = i \mod{2}$. This corresponds to the binary string $\mathbf{s} = 1010 \ldots$ of length $l$. This string is such that the diffsequence $\{i\}_{i=1}^{l}$ has $l-1$ hops, i.e. $\{i\}_{i=1}^{l} \in \Psi_{\mathbf{s}}(l-1)$ and $\psi_{\mathbf{s}}(l-1) = l$. Therefore, using Lemma~\ref{MainLemma} recursively, we can show that $\psi_{\mathbf{s}^{(l-1)}}(0) \le (3l^2-3l+2)/2$. That is, we show that the maximum possible length of a monochromatic diffsequence in $\mathbf{s}^{(l-1)}$ is at most $(3l^2-3l+2)/2$. The proof of Lemma~\ref{MainLemma} uses induction on the number of hops, exploiting the fact that diffsequences with $h$ hops, i.e., diffsequences in $\Psi_{\mathbf{s}^{(1)}}(h)$, are a concatenation of diffsequences with $h-1$ hops and monochromatic diffsequences. For an expanded string, we can therefore inductively apply Corollaries \ref{A B split cor} and \ref{pos bound cor}.

\begin{lemma}\label{MainLemma}
Let $\mathbf{s}$ be a binary string and let $\mathbf{s}^{(1)}$ denote the expansion of $\mathbf{s}$. Then 
\begin{equation}
    \psi_{\mathbf{s}^{(1)}}(h) \le \psi_{\mathbf{s}}(h+1) + 3h + 2. 
\end{equation}

\end{lemma}

\begin{proof}
Using induction on $h$, we will show that for any diffsequence in $\Psi_{\mathbf{s^{(1)}}}(h)$ of length $\ell$, there is a corresponding diffsequence in $\Psi_{\mathbf{s}}(h+1)$ with length at least $\ell-3h-2$. To prove the base case, $h=0$, let $ \mathcal{P}_0 =\{ p_i\}_{i=1}^{\ell} \in \Psi_{\mathbf{s}^{(1)}}(0)$ be a monochromatic diffsequence in $\mathbf{s}^{(1)}$ of length $\ell$. Without loss of generality, assume that $ \mathcal{P}_0$ has color $c=1$ in $\mathbf{s}^{(1)}$, i.e., $s^{(1)}_{p_i} = 1$, for $1 \le i \le \ell$. Now, consider the diffsequence $\text{pos}(\mathcal{P}_0) = \{\text{pos}(p_i)\}_{i=1}^{\ell}$ in $\mathbf{s}$. By Corollary~\ref{A B split cor}, we know that the substring $\{s_{\text{pos}(x_i)}\}_{i=1}^{\ell}$ can be split into at most two contiguous monochromatic substrings, the first with the color $1-c=0$ and the other with $c=1$. Let $\mathcal{P}_{\text{left}}$ and $\mathcal{P}_{\text{right}}$ be (possibly empty) subsequences of $ \mathcal{P}_0$ such that $\text{pos}(\mathcal{P}_{\text{left}})$ has color $0$ in $\mathbf{s}$, and $\text{pos}(\mathcal{P}_{\text{right}})$ has color $1$ in $\mathbf{s}$. Using Corollary~\ref{pos bound cor}, note that $|\text{pos}(\mathcal{P}_{\text{left}})|\ge|\mathcal{P}_{\text{left}}|-1$ and $|\text{pos}(\mathcal{P}_{\text{right}})|\ge|\mathcal{P}_{\text{right}}|-1$. Because $\text{pos}(\mathcal{P}_0)$ is the union of $\text{pos}(\mathcal{P}_{\text{left}})$ and $\text{pos}(\mathcal{P}_{\text{right}})$, and both $\text{pos}(\mathcal{P}_{\text{left}})$ and $\text{pos}(\mathcal{P}_{\text{right}})$ are disjoint, we get that $|\text{pos}(\mathcal{P}_0)|=|\text{pos}(\mathcal{P}_{\text{left}})|+|\text{pos}(\mathcal{P}_{\text{right}})|\ge |\mathcal{P}_{\text{left}}|-1+|\mathcal{P}_{\text{left}}|-1=|\mathcal{P}|-2$.

To prove the inductive hypothesis, assume that the result is true for all $k\leq h-1$. Let $\mathcal{P}_h=\{p_i\}^\ell_{i=1}$ be a diffsequence in $\mathbf{s}^{(1)}$ with at most $h$ hops. If $\mathcal{P}_h$ has at most $h-1$ hops, i.e. $\mathcal{P}_h \in \Psi_{\mathbf{s}^{(1)}}(h-1)$, then using the inductive hypothesis for $k = h-1$, we are done. Now, let us consider the case where $\mathcal{P}_h$ has exactly $h \ge 1$ hops. Let $p_i\to p_{i+1}$ be the first time there is a hop in $\mathbf{s}^{(1)}$. Let $\mathcal{P}_{\text{left}}$ be the first $i$ elements of $\mathcal{P}_h$ and $\mathcal{P}_{\text{right}}$ be the remaining part. Analogous to the proof of the base case, we split $\mathcal{P}_h$ into two contiguous subsequences $\mathcal{P}_{\text{left}}$ and $\mathcal{P}_{\text{right}}$ such that the hop between $\mathcal{P}_{\text{left}}$ and $\mathcal{P}_{\text{right}}$ is the first hop in $\mathcal{P}_h$. This means that $\mathcal{P}_{\text{left}}$ is monochromatic in $\mathbf{s}^{(1)}$ and $\mathcal{P}_{\text{right}}$ has at most $h-1$ hops in $\mathbf{s}^{(1)}$. Notice that $\text{pos}(\mathcal{P}_h)$ is the union of $\text{pos}(\mathcal{P}_{\text{left}})$ and $\text{pos}(\mathcal{P}_{\text{right}})$. Furthermore, $\text{pos}(\mathcal{P}_{\text{left}})$ and $\text{pos}(\mathcal{P}_{\text{right}})$ intersect in at most one element. Hence, we can lower bound $|\mathcal{P}_h|$ by $|\text{pos}(\mathcal{P}_{\text{left}})|+|\text{pos}(\mathcal{P}_{\text{right}})|-1$. As $\mathcal{P}_{\text{left}}\in \Psi_{\mathbf{s}^{(1)}}(0)$ and $\mathcal{P}_{\text{right}}\in \Psi_{\mathbf{s}^{(1)}}(h-1)$, we can apply the induction hypothesis to both $\mathcal{P}_{\text{left}}$ and $\mathcal{P}_{\text{right}}$ to obtain the following. $|\text{pos}(\mathcal{P}_h)| \ge |\text{pos}(\mathcal{P}_{\text{left}})|+|\text{pos}(\mathcal{P}_{\text{right}})|-1 \ge |(\mathcal{P}_{\text{left}})|-2+|(\mathcal{P}_{\text{right}})|-3(h-1)-2-1 \ge |\mathcal{P}_h|-3h-2.$
\end{proof}
\noindent
\textbf{Proof of Theorem 1.} 
We now explicitly construct a coloring of length $l \cdot4^{l-1}$ with the length of the maximum diffsequence at most $(3l^2-3l+2)/2$. We start with the binary string $\mathbf{s}$ of length $l$ consisting of alternating ones and zeros, i.e, the one corresponding to the coloring $\chi(i)\equiv i\pmod{2}$. This ensures $\psi_{\mathbf{s}}(l-1)=l$. By Lemma~\ref{MainLemma}, $\psi_{\mathbf{s}^{(i)}}(l-1-i)\le \psi_{\mathbf{s}^{(i-1)}}(l-1-i+1)+3(l-i-1)+2$ for all $1\le i\le l-1$.
So, summing over all $1\le i\le l-1$ we get
\begin{equation*}
    \psi_{\mathbf{s}^{(l-1)}}(0)\le l+\frac{3(l-1)(l-2)}{2}+2(l-1)=\frac{3l^2-3l+2}{2}.
\end{equation*}
For any positive integer $k>\frac{3l^2-3l+2}{2}$, we obtain a coloring of length $l\cdot 4^{l-1}$ that avoids a monochromatic diffsequence of length $k$, therefore, implying that $\Delta(D,k)\ge l\cdot4^{l-1}$. Hence, choosing the largest such $l=\left\lfloor\sqrt{\frac{2(k-1)}{3}+\frac{1}{4}}+\frac{1}{2}\right\rfloor$, we obtain $\Delta(D,k)\ge\left(\sqrt{\frac{8k-5}{12}}-\frac12\right)2^{\left(\sqrt{\frac{8k-5}{3}}-3\right)}$.

\end{document}